\documentclass[12pt]{amsart}
\usepackage{amsfonts}
\usepackage{amssymb}
\usepackage{times}
\newtheorem{thm}{Theorem}[section]
\newtheorem{lemma}[thm]{Lemma}
\newtheorem{cor}[thm]{Corollary}
\newtheorem{prop}[thm]{Proposition}

\newcounter{other}            
\newtheorem{otherth}[other]{Theorem}              
\newtheorem{otherl}[other]{ Lemma}        

\renewcommand{\Re}{\mbox{Re}}
\def\B{\mathcal{B}}
\def\D{\mathbb{D}}
\def\R{\mathbb R}
\def\T{\mathbb{T}}
\def\C{\mathbb{C}}
\def\Q{\mathcal{Q}}

\def\qp{\Q_p}
\def\qs{\Q_s}
\def\qst{\Q_s(\T)}
\def\f{\frac}
\def \qpst {\Q_s^p(\T)}
\def \qpsd {\Q_s^p(\D)}
\numberwithin{equation}{section}

\begin{document}

\title[Boundary multipliers]
{Boundary multipliers of a family of M\"obius invariant function spaces}
\author{Guanlong Bao and  Jordi Pau}
\address{Guanlong Bao\\
Department of Mathematics\\
    Shantou University\\
    Shantou, Guangdong 515063, China}
\email{glbaoah@163.com}

\address{Jordi Pau\\
Departament de Matem\'atica Aplicada i Analisi\\
    Universitat de Barcelona\\
    08007 Barcelona, Spain}
\email{jordi.pau@ub.edu}

\thanks{G. Bao is supported in part by NSF of China (No. 11371234).  J. Pau was supported by SGR grant 2014SGR289 (Generalitat de Catalunya) and DGICYT grant MTM2011-27932-$C$02-01 (MCyT/MEC)}
\subjclass[2010]{30H25, 30J10, 46E15}

\keywords {pointwise multipliers, Carleson measures, Blaschke products, $\qpst$ spaces}

\begin{abstract}
For $1<p<\infty$ and $0<s<1$,  let  $\qpst$ be the space of those functions $f$ which belong to
 $ L^p(\T)$ and satisfy
$$
\sup_{I\subseteq \T}\f{1}{|I|^s}\int_I\int_I\f{|f(\zeta)-f(\eta)|^p}{|\zeta-\eta|^{2-s}}|d\zeta||d\eta|<\infty,
$$
where $|I|$ is the length of an arc $I$ of the unit circle $\T$ . In this paper, we give a complete description of multipliers between  $\qpst$ spaces. The spectra  of multiplication operators on $\qpst$ are  also obtained.
\end{abstract}

\maketitle

\section{Introduction}

An important problem of  studying  function  spaces  is to characterize  the pointwise multipliers
of such spaces.   For Banach  function spaces $X$ and $Y$,  denote by $M(X, Y)$ the class   of all pointwise  multipliers from $X$ to $Y$. Namely,
\[
 M(X, Y)=\{f:  fg\in Y \ \text{for all }\  g\in X\}.
\]
If $X=Y$, we  just write $M(X, Y)$ as $M(X)$ for the collection  of  multipliers  of $X$. For any $g\in M(X,Y)$, denote by $M_g$ the multiplication operator induced by $g$, that is, $M_g(f)=gf$. By the closed graph theorem, $M_g$ is a bounded operator. In this paper we characterize the pointwise multipliers between a certain family of function spaces on the unit circle. These spaces appear in a natural way as the boundary values of a certain family of analytic M\"{o}bius invariant spaces on the disk \cite{Zha1} that has been attracted much attention recently.

Denote by $\T$ the boundary of  the unit disk $\D$ in the complex plane $\C$.  Let $H(\D)$ be  the space  of all analytic functions on $\D$  and let $H^\infty$ be the class of bounded analytic functions on $\D$. For $1<p<\infty$ and $s\ge 0$, consider the analytic Besov type space $B_ p(s)$ consisting of those functions
$f\in H(\D)$ with
$$
\|f\|_{B_p(s)}=\left(\int_{\D}|f'(z)|^p\,(1-|z|^2)^{p-2+s}\,dA(z)\right)^{1/p} <\infty,
$$
where $dA$ denotes the   Lebesgue measure on $\D$. A norm in $B_ p(s)$ is given by $|f(0)|+\|f\|_{B_ p(s)}$. For $a\in \D$, let
$$
\sigma_a(z)=\frac{a-z}{1-\overline{a}z}, \qquad z\in \D,
$$
be  a M\"obius transformation of $\D$. The space $B_ p=B_ p(0)$ is M\"{o}bius invariant in the sense that $\|f\circ \sigma_ a\|_{B_ p}=\|f\|_{B_ p}$, and $B_ 2$ is the classical Dirichlet space. For $s>0$ and $1<p<\infty$, we denote by $\qpsd$ the M\"{o}bius invariant space generated by $B_ p(s)$, that is, $f\in \qpsd$ if $f\in H(\D)$ and
\[
\|f\|^p_{\qpsd} =\sup_{a\in \D} \|f\circ \sigma_ a \|^p_{B_ p(s)}<\infty.
\]
The space $\qpsd$ coincides with $F(p,p-2,s)$ where $F(p,q,s)$ is the family of function spaces studied in \cite{Zha1} and \cite{Rat}. In particular, for $s>1$ the spaces  $\qpsd$ are the same and equal to the Bloch space  $\B$ (the ``maximal" M\"{o}bius invariant space) which consists of all functions $f\in H(\D)$ with
$$
\|f\|_{\B}=\sup_{z\in \D}(1-|z|^2)|f'(z)|<\infty.
$$
When $p=2$, one has $\mathcal{Q}^2_ s(\D)=\Q_ s(\D)$ the holomorphic $Q$ spaces introduced in \cite{AXZ} and widely studied in the monographs \cite{Xi3,Xi5}.
 In particular, $\Q_1(\D)=BMOA$, the space of analytic functions with bounded mean oscillation  \cite{Gi}.
M. Ess\'en and J. Xiao \cite{EX} gave that if $0<s<1$ and $f$ is in the Hardy space $H^2$, then $f\in  \Q_s(\D)$ if and only if $f\in \qst$, the space of functions $f\in L^2(\T)$ with
$$
\|f\|^2_{\qst}=\sup_{I\subseteq \T}\f{1}{|I|^s}\int_I\int_I\f{|f(\zeta)-f(\eta)|^2}{|\zeta-\eta|^{2-s}}|d\zeta||d\eta|<\infty,
$$
where $|I|$ is the length of an arc $I$ of the unit circle $\T$ (a version of these spaces for several real variables was studied in \cite{EJPX}).
If $s>1$, J. Xiao \cite{Xi2} pointed out that $\qst$ are equal to $BMO(\T)$, the space of bounded mean oscillation on $\T$. For $p>1$, via the John-Nirenberg inequality (see \cite{Gi, JN}), one gets
$$
\|f\|_{BMO(\T)}^p \thickapprox \sup_{I\subseteq \T}\f{1}{|I|}\int_I|f(\zeta)-f_I|^p|d\zeta|,
$$
 where $f_I$ is the average of $f$ over $I$, that is
$$
f_I=\f{1}{|I|}\int_If(\zeta)|d\zeta|.
$$
In view of that it is natural to consider, for $1<p<\infty$ and $s>0$, the spaces $\qpst$ consisting of functions  $f\in L^p(\T)$ such that
\begin{equation}\label{Eq1-1}
\|f\|^p_{\qpst}=\sup_{I\subseteq \T}\f{1}{|I|^s}\int_I\int_I\f{|f(\zeta)-f(\eta)|^p}{|\zeta-\eta|^{2-s}}|d\zeta||d\eta|<\infty.
\end{equation}
A true norm in $\qpst$ is given by $\|f\|_{*,\qpst}=\|f\|_{L^p(\T)}+\|f\|_{\qpst}$. We are going to study these spaces, and we will see that if $f$ is in the Hardy space $H^p$, then $f\in \qpsd$ if and only if $f\in \qpst$. Also, we give a complete description of the pointwise multipliers $M(\mathcal{Q}^{p_ 1}_ s(\T),\mathcal{Q}^{p_ 2}_ r(\T))$ for $0<p_ 1,p_ 2<\infty$ and $0<s,r<1$.  It is worth mentioning that D. Stegenga \cite{Steg1} characterized the multipliers of bounded mean oscillation spaces on the unit circle (see also \cite{OF}), and  L. Brown and A.  Shields \cite{BS} described the pointwise multipliers of the Bloch space. A characterization of the pointwise multipliers $M(\Q_s(\D))$ was obtained in \cite{PP1} proving a conjecture stated in \cite{Xi1}. See \cite{GGP, OF, PO, Steg2, WY,  Zhu1} for more results on pointwise multipliers of function spaces.\\
\mbox{}
\\
The following is the main result of the paper.

\begin{thm}\label{mt1}
 Let $1<p_1, p_2<\infty$ and $0<s, r<1$. Then the following are true.
\begin{enumerate}
\item [(1)] If $p_1\leq p_2$ and $s\leq r$, then $f\in M(\Q_s^{p_1}(\T), \Q_r^{p_2}(\T))$ if and only if $f\in L^\infty(\T)$ and
\begin{equation}\label{Eq1-2}
\sup_{I\subseteq \T} \f{1}{|I|^r}\left(\log\f{2}{|I|}\right)^{p_2}\int_I\int_I \f{|f(\zeta)-f(\eta)|^{p_2}}{|\zeta-\eta|^{2-r}}|d\zeta||d\eta| <\infty.
\end{equation}
\item[(2)] Let  $p_1>p_2$ and $s\leq r$. If $\f{1-s}{p_1}>\f{1-r}{p_2}$, then $f\in M(\Q_s^{p_1}(\T), \Q_r^{p_2}(\T))$ if and only if $f\in L^\infty(\T)$ and $f$ satisfies \eqref{Eq1-2}. If $\f{1-s}{p_1}\leq\f{1-r}{p_2}$, then  $M(\Q_s^{p_1}(\T), \Q_r^{p_2}(\T))=\{0\}$.\\
\item[(3)] If $s>r$, then $M(\Q_s^{p_1}(\T), \Q_r^{p_2}(\T))=\{0\}$.
\end{enumerate}
\end{thm}
\mbox{}
\\
Note that, when $p_ 1=p_ 2=2$ and $s=r$, part (1) of Theorem \ref{mt1} proves Conjecture 2.5 stated in \cite{Xi1}. However, as seen in the proof, this conjecture is an immediate consequence of the results and methods in \cite{PP1}.\\
\mbox{}
\\
Next we give an application. Let $T$ be a bounded linear operator on a Banach space $X$. The spectrum of $T$ is defined as
$$
\sigma(T)=\{\lambda\in\C: \  T-\lambda E \text{ is not invertible}\},
$$
where $E$ is the identity operator on $X$. R. Allen and F. Colonna \cite{AC} gave  spectra of multiplication operators on the Bloch space. In this paper, we also consider the   spectra of multiplication operators on $\qpst$ spaces. For $f\in \qpst$, let $\mathcal R(f)$ be the essential range of $f$. Namely,  $\mathcal R(f)$ is the set of all $\lambda$ in $\C$ for which $\{\zeta\in \T: |f(\zeta)-\lambda|<\varepsilon\}$ has positive measure for every $\varepsilon>0$. By  \cite[p. 57]{Do}, if $f\in L^\infty(\T)$, then $\mathcal R(f)$ is a compact subset of $\C$.

\begin{thm} \label{t2}
Suppose $1<p<\infty$ and $0<s<1$. Let $f$ be the symbol of a bounded multiplication operator $M_f$ on $\qpst$ space. Then $\sigma(M_f)=\mathcal R(f)$.
\end{thm}

The paper is organized as follows. In Section 2, we give some preliminaries as well as basic properties of $\qpst$ spaces, such as inclusion relations or a characterization in terms of Carleson type measures. The proof of Theorem \ref{mt1} is given in Section 3. Particulary interesting is the proof of part (3), where we are in need to use the  results of A. Nagel, W. Rudin and J. Shapiro \cite{NRS},
on tangential boundary behavior of functions in weighted analytic Besov  spaces.   In Section 4, we prove Theorem \ref{t2}, and in the last section, we give the analytic versions of Theorems \ref{mt1} and \ref{t2}.

Throughout this  paper, for a positive number $\lambda$ and an arc $I\subseteq \T$, denote by $\lambda I$ the arc with the same center as $I$ and with the length $\lambda |I|$. The symbol $A\thickapprox B$ means that $A\lesssim B\lesssim
A$. We say that $A\lesssim B$ if there exists a constant $C$ such that $A\leq CB$.

\section{Preliminaries and basic properties}

An important tool to study function spaces is  Carleson type measures. Given an arc $I$ on $\T$, the
 Carleson sector $S(I)$ is given by
$$
S(I)=\{r\zeta \in \D: 1-\f{|I|}{2\pi}<r<1, \ \zeta\in I\}.
$$
For $s>0$, a positive Borel measure $\mu$ on $\D$ is said to be an $s$-Carleson measure if
$$
\sup_{I\subseteq \T} \f{\mu(S(I))}{|I|^s}<\infty.
$$
When $s=1$, we get the classical Carleson measures, characterizing when $H^p\subset L^p(\D,\mu)$, where
for $0<p<\infty$, $H^p$ denotes the classical Hardy space \cite{Du} of  functions  $f\in H(\D)$ for which
$$
\sup_{0<r<1} M_p(r, f)<\infty.
$$
Here
$$
M_p(r, f)= \left(\f{1}{2\pi}\int_0^{2\pi}|f(re^{i\theta})|^p d\theta \right)^{1/p}.
$$
By \cite{AXZ},  $\mu$ is an  $s$-Carleson measure  if and only if
\[
\sup_{a\in \D}\int_{\D}\left(\f{1-|a|^2}{|1-\bar{a}z|^2}\right)^sd\mu(z)<\infty.
\]
Because
\[
\|f\circ \sigma_ a \|^p_{B_ p(s)}=\int_{\D} |f'(z)|^p \,(1-|z|^2)^{p-2}\,(1-|\sigma_ a(z)|^2)^s \,dA(z),
\]
 we can immediately see that $f\in \qpsd$ if and only if $|f'(z)|^p\,(1-|z|^2)^{p-2+s}\,dA(z)$ is an $s$-Carleson measure.\\

A function  $f\in H(\D)$  is called an inner function if it is an $H^\infty$-function with radial limits of modulus one almost everywhere on the unit circle.
A sequence $\{a_k\}_{k=1}^\infty\subseteq \D$  is said to be a Blaschke sequence if
$$
\sum_{k=1}^\infty (1-|a_k|)<\infty.
$$
The above condition implies the convergence of the corresponding Blaschke product $B$ defined as
$$
B(z)=\prod_{k=1}^\infty \f{|a_k|}{a_k}\f{a_k-z}{1-\overline{a_k}z}.
$$
We also  need the characterizations of inner functions  in $\qpsd$ spaces.
  M. Ess\'en and J. Xiao \cite{EX} characterized inner functions in $\Q_s(\D)$ spaces. Later,
F. P\'erez-Gonz\'alez and J. R\"atty\"a \cite{PR} described inner functions in $\qpsd$ spaces as follows.

\begin{otherth} \label{T-Inner}
Let $0<s<1$ and $p>\max\{s, 1-s\}$. Then an inner function belongs to $\qpsd$ if and only if it is a Blaschke product associated with a sequence $\{z_k\}_{k=1}^\infty\subseteq \D$ which satisfies that
$\sum_k (1-|z_k|)^s \delta_{z_k}$ is an $s$-Carleson measure, that is
$$
\sup_{a\in \D}\sum_{k=1}^\infty  \left(1-|\sigma_a(z_k)|^2\right)^s<\infty.
$$
\end{otherth}

The proof of Theorem \ref{mt1} will also use the Rademacher functions $\{r_j(t)\}_{j=0}^\infty$ defined by
$$
r_0(t)=\begin{cases}
1, & 0<t<\f{1}{2}, \\
-1, & \f{1}{2}<t<1,\\
0, & t=0,\ \f{1}{2},\ 1.
 \end{cases}
$$
$$
r_n(t)=r_0(2^nt),\ \ n=1,\ 2,\ \cdots.
$$
See \cite[Chapter V, Vol. I]{Zy} or \cite[Appendix A]{Du} for  properties of these functions. In particular, we will  use Khinchine's inequality.

\begin{otherth}\label{KI} (Khinchine's inequality). If $\{c_k\}_{k=1}^\infty\in \ell^2$, then the series $\sum_{k=1}^\infty c_kr_k(t)$ converges almost everywhere. Furthermore, for $0<p<\infty$ there exist positive constants $A_p$, $B_p$ such that for every sequence $\{c_k\}_{k=1}^\infty\in \ell^2$ we have
$$
A_p\left(\sum_{k=1}^\infty |c_k|^2\right)^{\f{p}{2}}
\leq \int_0^1\left|\sum_{k=1}^\infty c_kr_k(t)\right|^pdt
\leq B_p\left(\sum_{k=1}^\infty |c_k|^2\right)^{\f{p}{2}}.
$$
\end{otherth}

For $1<p<\infty$ and $0<s<1$, the following result shows that we can regard $\qpst$ as a Banach space of functions
modulo constants.

\begin{lemma} \label{L1}
Let $1<p<\infty$ and $0<s<1$. Then $\qpst \subseteq BMO(\T)$. Furthermore,   $\qpst$ is complete with respect to (1.1).
\end{lemma}

\begin{proof}
Let $f\in \qpst$. For any arc $I\subseteq \T$, it follows from H\"older's inequality that
\[
\begin{split}
\f{1}{|I|}\int_I|f(e^{it})-f_I|\,dt &
\leq \f{1}{|I|^2}\int_I\int_I|f(e^{it})-f(e^{i\theta})| \,d\theta \,dt
\\
&\leq \left(\f{1}{|I|^s}\int_I\int_I\f{|f(e^{it})-f(e^{i\theta})|^p}{|e^{it}-e^{i\theta}|^{2-s}}\,d\theta \, dt\right)^{1/p},
\end{split}
\]
because $|e^{it}-e^{i\theta}|\le |I|$ when $e^{it}$ and $e^{i\theta}$ are in $I$. Thus $\qpst \subseteq BMO(\T)$ with $\|f\|_{BMO(\T)}\lesssim  \|f\|_{\qpst}$.
Now let $\{f_m\}$ be a Cauchy sequence in $\qpst$. Then it is also a Cauchy sequence in $BMO(\T)$. Hence $f_m\rightarrow f$  in $BMO(\T)$ for some $f$ in $BMO(\T)$ and there exists a subsequence $\{f_{m_k}\}\subseteq \{f_m\}$ such that
$
\lim_{k\rightarrow\infty}f_{m_k}(e^{it})=f(e^{it}),$ for a.e.  $e^{it}\in I.$
 By Fatou's lemma, it follows easily that
$$
\|f_{m_k}-f\|_{\qpst}\leq\liminf_{l\rightarrow\infty}
\|f_{m_l}-f_{m_k}\|_{\qpst},
$$
 which implies that $f_{m_k}\rightarrow f$ in $\qpst$.
Since
 $$
 \|f_k-f\|_{\qpst}\leq\|f_{m_k}-f\|_{\qpst}+
 \|f_{m_k}-f_k\|_{\qpst},
 $$
 this finishes  the proof.
\end{proof}

\subsection{Characterizations of $\qpst$ spaces}\label{s3}

To prove our main results in this paper, some characterizations of  $\qpst$ spaces are necessary.
Given $f\in L^1(\T)$, let $\widehat{f}$ be the Poisson extension of $f$, that is,
$$
\widehat{f}(z)=\f{1}{2\pi}\int_0^{2\pi}f(e^{i\theta})\f{1-|z|^2}{|e^{i\theta}-z|^2}d\theta,\qquad z\in\D.
$$
We will characterize $\qpst$ spaces in terms of Carleson type measures. Before doing that, we state and prove some  auxiliary results.

\begin{lemma} \label{EL} Suppose  $p>1$ and $0<s<1$. Let $f\in L^p(\T)$ and let $F\in C^1(\D)$ with
$
\lim_{r\rightarrow1}F(re^{it})=f(e^{it})$  for  a.e. $e^{it}\in \T.$ For any arc $I\subset \T$, we have
\[
\int_{I}\int_{I} \frac{|f(e^{is})-f(e^{i\theta})|^p}{|e^{is}-e^{i\theta}|^{2-s}}\,d\theta\,ds \lesssim \int_{S(3I)} |\nabla F(z)|^p(1-|z|^2)^{p-2+s}dA(z).
\]
In particular, if $|\nabla F(z)|^p(1-|z|^2)^{p-2+s}dA(z)$ is an $s$--Carleson measure, then $f\in \qpst$.
\end{lemma}

\begin{proof} We use an argument used in \cite{NX} and \cite[p. 1294]{ABP}. After a change of variables, it is easy to see (it is done in the same way as in \cite{NX} or \cite[Chapter 7]{Xi3} where the case $p=2$ was obtained) that
$$
\int_{I}\int_{I} \frac{|f(e^{is})-f(e^{i\theta})|^p}{|e^{is}-e^{i\theta}|^{2-s}}\,d\theta\,ds \lesssim\int_0^{|I|} \f{1}{t^{2-s}}\left(\int_I |f(e^{i(\theta+t)})-f(e^{i\theta})|^p\,d\theta\right)\,dt.
$$
For any  arc $I$ with $|I|<\f{\pi}{3}$ and $t\in (0, |I|)$, set  $r=1-t$. Then
\[
\begin{split}
|f(e^{i(\theta+t)})&-f(e^{i\theta})|
\\
&\leq\big |f(e^{i(\theta+t)})-F(re^{i(\theta+t)})\big | +\big|F(re^{i(\theta+t)})-F(re^{i\theta})\big |+\big |F(re^{i\theta})-f(e^{i\theta})\big |
\\
&\leq\int_r^1|\nabla F(xe^{i(\theta+t)})|dx+\int_0^t|\nabla F(re^{i(\theta+u)})|du+\int_r^1|\nabla F(xe^{i\theta})|dx.
\end{split}
\]
Since $p>1$, we can use Minkowski's inequality to obtain
\[
\begin{split}
\int_I |f&(e^{i(\theta+t)})-f(e^{i\theta})|^pd\theta
\lesssim \left(\int_r^1 \left(\int_I    |\nabla F(xe^{i(\theta+t)})|^pd\theta\right)^{1/p}  dx\right)^p
\\
&\quad+
\left(\int_0^t \left(\int_I    |\nabla F(re^{i(\theta+u)})|^pd\theta\right)^{1/p}  du\right)^p
+\left(\int_r^1 \left(\int_I    |\nabla F(xe^{i\theta})|^pd\theta\right)^{1/p}  dx\right)^p
\\
&\lesssim \left(\int_r^1 \left(\int_{3I}    |\nabla F(xe^{i\theta})|^pd\theta\right)^{1/p}  dx\right)^p +
\left(\int_0^t \left(\int_I    |\nabla F(re^{i(\theta+u)})|^pd\theta\right)^{1/p}  du\right)^p\\
\\
&=(J_ 1)+(J_ 2).
\end{split}
\]
For  $p>1$ and  $0<s<1$, applying the Hardy inequality (see \cite[p. 272]{Stei}) gives
\begin{eqnarray*}
\int_0^{|I|} \f{1}{t^{2-s}}\,(J_ 1) \,dt&=&\int_0^{|I|} \f{1}{t^{2-s}}\left(\int_0^t \left(\int_{3I}    |\nabla F((1-y)e^{i\theta})|^pd\theta\right)^{1/p}  dy\right)^p dt
\\
&\leq&\left(\f{p}{1-s}\right)^p\int_0^{|I|}\left(\int_{3I} |\nabla F((1-y)e^{i\theta})|^pd\theta\right) y^{p-2+s}dy\\
&=&\left(\f{p}{1-s}\right)^p\int_{1-|I|}^1\left(\int_{3I} |\nabla F(xe^{i\theta})|^pd\theta\right) (1-x)^{p-2+s}dx\\
&\lesssim& \int_{S(3I)} |\nabla F(z)|^p(1-|z|)^{p-2+s}dA(z).
\end{eqnarray*}
We also have
\begin{eqnarray*}
\int_0^{|I|} \f{1}{t^{2-s}}\,(J_ 2) \,dt
&\leq&\int_0^{|I|} \f{1}{t^{2-s}}\left(\int_0^t \left(\int_{3I}    |\nabla F(re^{i\theta})|^pd\theta\right)^{1/p}  du\right)^p dt
\\
&=&\int_0^{|I|} t^{p-2+s}\left(\int_{3I}    |\nabla F((1-t)e^{i\theta})|^pd\theta\right) dt
\\
&\lesssim& \int_{S(3I)} |\nabla F(z)|^p(1-|z|)^{p-2+s}dA(z).
\end{eqnarray*}
Combining  the above estimates, we obtain the desired result.
\end{proof}

We also need the following result which is a generalization of    D. Stegenga's  estimate in \cite{Steg2}.

\begin{lemma} \label{St-L}
For $p>1$ and $0<s<1$, let $f\in L^p(\T)$ and let $I$, $J$ be two arcs on $\T$ centered at $e^{i\theta_0}$ with $|J|\geq 3|I|$.  Then there exists a constant $C$ depending only on $p$ and $s$ such that
\begin{eqnarray*}
&~&\int_{S(I)}|\nabla\widehat{f}(z)|^p(1-|z|)^{p-2+s}dA(z)\\
&\leq& C \left[\int_J\int_J\f{|f(e^{i\theta})-f(e^{it})|^p}{|e^{i\theta}-e^{it}|^{2-s}}d\theta dt+|I|^{p+s}\left(\int_{|t|\geq \f{1}{3}|J|}|f(e^{i(t+\theta_0)})-f_J|\f{dt}{t^2}\right)^p\right].
\end{eqnarray*}
\end{lemma}

\begin{proof} Without loss of generality, one  may assume  that $\theta_0=0$. Following Stegenga, we let $\phi$ be a function with $0\leq \phi\leq 1$ such that $\phi=1$ on $\f{1}{3} J$, $\text{supp}\ \phi \subseteq \f{2}{3} J$, and
\begin{equation}\label{Ephi}
|\phi(e^{i\theta})-\phi(e^{it})|\lesssim\f{|e^{i\theta}-e^{it}|}{|J|}
\end{equation}
for all $\theta$, $t\in[0,\ 2\pi)$.  Now we write
$$
f=(f-f_J)\phi+(f-f_J)(1-\phi)+f_J=f_1+f_2+f_3.
$$
Since $f_3$ is constant, we have $\nabla\widehat{f_3}=0$.
For $z=re^{i\theta}$ in the Carleson box $S(I)$,
$$
|\nabla\widehat{f_2}(z)|
\lesssim\int_0^{2\pi} \f{|f_2(e^{it})|}{(1-r)^2+(\theta-t)^2}dt
\lesssim\int_{|t|\geq \f{1}{3}|J|} |f(e^{it})-f_J|\f{dt}{t^2}
$$
and hence
$$
\int_{S(I)}|\nabla\widehat{f_2}(z)|^p(1-|z|)^{p-2+s}dA(z)
\lesssim |I|^{p+s}\left(\int_{|t|\geq \f{1}{3}|J|}|f(e^{it})-f_J|\f{dt}{t^2}\right)^p.
$$
For the integral over $S(I)$ of $|\nabla\widehat{f_1}|^p$, replacing $S(I)$ with the unit  disc $\D$ and using Proposition 4.2 in \cite{ABP}, we obtain
\[
\begin{split}
\int_{S(I)}|&\nabla\widehat{f_1}(z)|^p(1-|z|)^{p-2+s}dA(z)
\\
&\lesssim \int_\T\int_\T\f{|f_1(e^{i\theta})-f_1(e^{it})|^p}{|e^{i\theta}-e^{it}|^{2-s}}d\theta dt
\\
&\thickapprox \int_{e^{i\theta}\in J}\int_{e^{it}\in J}\f{|f_1(e^{i\theta})-f_1(e^{it})|^p}{|e^{i\theta}-e^{it}|^{2-s}}d\theta dt
+\int_{e^{i\theta}\not\in J}\int_{e^{it}\in \f{2}{3}J}\f{|f_1(e^{i\theta})-f_1(e^{it})|^p}{|e^{i\theta}-e^{it}|^{2-s}}d\theta dt
\\
&\quad +\int_{e^{it}\not\in J}\int_{e^{i\theta}\in \f{2}{3}J}\f{|f_1(e^{i\theta})-f_1(e^{it})|^p}{|e^{i\theta}-e^{it}|^{2-s}}d\theta dt
\,\thickapprox \, T_1+T_2+T_3.
\end{split}
\]
For estimating $T_1$, we see  first that, due to the condition \eqref{Ephi}, for $e^{i\theta},\ e^{it}\in\T$,
$$
|f_1(e^{i\theta})-f_1(e^{it})|\lesssim |f(e^{i\theta})-f(e^{it})|+|J|^{-1}|e^{i\theta}-e^{it}||f(e^{it})-f_J|.
$$
By H\"older's inequality, we  deduce that
\begin{eqnarray*}
\f{1}{|J|^p}\int_J\int_J\f{|f(e^{it})-f_J|^p}{|e^{i\theta}-e^{it}|^{2-s-p}}d\theta dt
&=&\f{1}{|J|^p}\int_J|f(e^{it})-f_J|^p\left(\int_J|e^{i\theta}-e^{it}|^{p-2+s}d\theta\right) dt
\\
&\lesssim&\f{1}{|J|^{1-s}}\int_J|f(e^{it})-f_J|^pdt\\
&\lesssim& \f{1}{|J|^{2-s}}\int_J\int_J|f(e^{i\theta})-f(e^{it})|^p d\theta dt\\
&\lesssim&\int_J\int_J\f{|f(e^{i\theta})-f(e^{it})|^p}{|e^{i\theta}-e^{it}|^{2-s}}d\theta dt.
\end{eqnarray*}
Thus,
$$
T_1 \lesssim \int_J\int_J\f{|f(e^{i\theta})-f(e^{it})|^p}{|e^{i\theta}-e^{it}|^{2-s}}d\theta dt.
$$
For $T_2$,  using that $f_1(e^{i\theta})=0$ for $e^{i\theta}\not\in J$, one gets
\[
\begin{split}
T_2&=\int_{e^{i\theta}\not\in J}\int_{e^{it}\in \f{2}{3}J}\f{|f_1(e^{i\theta})-f_1(e^{it})|^p}{|e^{i\theta}-e^{it}|^{2-s}}d\theta dt
\leq\int_{e^{i\theta}\not\in J}\int_{e^{it}\in \f{2}{3}J}\f{|f(e^{it})-f_J|^p}{|e^{i\theta}-e^{it}|^{2-s}}d\theta dt
\\
&\lesssim \f{1}{|J|^{1-s}}\int_J|f(e^{it})-f_J|^pdt
\lesssim \int_J\int_J\f{|f(e^{i\theta})-f(e^{it})|^p}{|e^{i\theta}-e^{it}|^{2-s}}d\theta dt.
\end{split}
\]
The estimate of $T_3$ is similar to $T_2$.
The above inequalities implies the lemma.
\end{proof}

The following theorem characterizes $\qpst$ spaces in terms of Carleson type measures. It  generalizes the corresponding result of $\qst$ spaces in \cite{NX}.

\begin{thm}\label{t3} Let $p>1$ and $0<s<1$. Suppose $f\in L^p(\T)$. The following conditions are equivalent.
\begin{enumerate}
\item [(a)] $f\in \qpst$.
\item[(b)] $|\nabla \widehat f(z)|^p(1-|z|^2)^{p-2+s}dA(z)$ is an $s$-Carleson measure.
\item[(c)]
$$
\sup_{a\in\D}\int_\T\int_\T\f{|f(\zeta)-f(\eta)|^p}{|\zeta-\eta|^{2-s}}\left(\f{1-|a|^2}{|\zeta-a||\eta-a|}\right)^s|d\zeta||d\eta|<\infty.
$$
\end{enumerate}
\end{thm}

\begin{proof} We first show that  (b) is equivalent to (c). By Proposition 4.2 in \cite{ABP}, one gets
\begin{equation}\label{Eq-C1}
\int_{\D}|\nabla \widehat f(z)|^p(1-|z|^2)^{p-2+s}dA(z)\thickapprox \int_0^{2\pi}\int_0^{2\pi}\f{|f(e^{i\theta})-f(e^{it})|^p}{|e^{i\theta}-e^{it}|^{2-s}}d\theta dt
\end{equation}
for all $f\in L^p(\T)$. Note that $\widehat{f\circ \sigma_a}=\widehat{f}\circ \sigma_a$ for any $a\in \D$.
Replacing $f$ by $f\circ \sigma_a$  in the above formula and making a change of variables, we get
\[
\begin{split}
\int_{\D} &|\nabla \widehat f(z)|^p (1-|z|^2)^{p-2+s}\left(\f{1-|a|^2}{|1-\bar{a}z|^2}\right)^s \, dA(z)
\\
&\thickapprox\int_0^{2\pi}\int_0^{2\pi}\f{|f(e^{i\theta})-f(e^{it})|^p}{|e^{i\theta}-e^{it}|^{2-s}} \left(\f{1-|a|^2}{|1-\overline{a}e^{i\theta}|
|1-\overline{a}e^{it}|}\right)^s\,d\theta \, dt.
\end{split}
\]
This  gives that  (b) is equivalent to (c).

By Lemma \ref{EL}, we see that (b) implies (a). Next we  verify that (a) implies (b). Let  $f\in \qpst$. Then $f$ is also in $BMO(\T)$  by Lemma \ref{L1}. For an arc $I$ centered at $e^{i\theta_0}$, let $J=3I$.  Then Lemma \ref{St-L} gives
\begin{eqnarray*}
&~&\int_{S(I)}|\nabla\widehat{f}(z)|^p(1-|z|)^{p-2+s}dA(z)\\
&\lesssim& \int_J\int_J\f{|f(e^{i\theta})-f(e^{it})|^p}{|e^{i\theta}-e^{it}|^{2-s}}d\theta dt+|I|^{p+s}\left(\int_{|t|\geq \f{1}{3}|J|}|f(e^{i(t+\theta_0)})-f_J|\f{dt}{t^2}\right)^p.
\end{eqnarray*}
Since, by \cite[p. 71]{Xi3},
$$
|J|\int_{|t|\geq \f{1}{3}|J|}|f(e^{i(t+\theta_0)})-f_J|\f{dt}{t^2}\lesssim \|f\|_{BMO(\T)},
$$
we have
$$
\sup_{I\subseteq \T}\f{1}{|I|^s}\int_{S(I)}|\nabla\widehat{f}(z)|^p(1-|z|)^{p-2+s}dA(z)\lesssim \|f\|_{\qpst}^p+\|f\|_{BMO(\T)}^p.
$$
The proof is complete.
\end{proof}
\vskip 3mm
If $p>1$ and $0<s<1$, then  $\qpsd\subseteq BMOA\subseteq H^2$. Thus functions in $\qpsd$ have boundary values.
As noticed before, an analytic function $f$ belongs to $\qpsd$ if and only if $|f'(z)|^p(1-|z|^2)^{p-2+s}dA(z)$ is an $s$-Carleson measure.
Combining this with Theorem \ref{t3}, one gets the following result immediately.

\begin{cor}\label{C1}
 Let $p>1$ and $0<s<1$. Suppose $f\in H^1$. Then $f\in \qpsd$ if and only if $f\in \qpst$.
\end{cor}

\noindent{\bf {\sl Remark 1}.}\ \   Let $p>1$ and $0<s<1$.  We say that $f\in L_s^p$ if $f\in L^p(\T)$ and
$$
\|f\|_{L_s^p}^p=\int_0^{2\pi}\int_0^{2\pi}\f{|f(e^{it})-f(e^{i\theta})|^p}{|e^{it}-e^{i\theta}|^{2-s}}d\theta dt<\infty.
$$
The condition (c) of Theorem \ref{t3} gives that $f\in \qpst$ if and only if
$$
\sup_{a\in \D}\|f\circ \sigma_a\|_{L_s^p}<\infty.
$$
Thus if we set
$$
|||f|||_{\qpst}=\sup_{a\in \D}\|f\circ \sigma_a\|_{L_s^p},
$$
then $\qpst$ is a M\"obius invariant space in the sense of that
$$
|||f|||_{\qpst}=|||f\circ \sigma_a|||_{\qpst}
$$
for any $f\in \qpst$ and $a\in \D$.

\subsection{Inclusion relations}

Applying  Theorem \ref{t3} and Corollary \ref{C1},  we  can obtain a complete picture on the inclusion relations between different $\qpst$ spaces.

\begin{thm}\label{inclusions} Let $1<p_1, p_2<\infty$ and $0<s, r<1$.
\begin{enumerate}
\item [(1)] If $p_1\leq p_2$, then $\Q_s^{p_1}(\T)\subseteq \Q_r^{p_2}(\T)$ if and only if $s\leq r$.
\item [(2)]If $p_1>p_2$, then $\Q_s^{p_1}(\T)\subseteq \Q_r^{p_2}(\T)$ if and only if $\f{1-s}{p_1}> \f{1-r}{p_2}$.
\end{enumerate}
\end{thm}

\begin{proof} We first consider the inclusion relation between the analytic spaces $\Q_s^{p_1}(\D)$ and $\Q_r^{p_2}(\D)$. Note that  $\Q_s^{p_1}(\D)$ is a subset of the Bloch space $\B$.   Let $f\in \Q_s^{p_1}(\D)$. If $p_1\leq p_2$ and  $s\leq r$, then
\begin{eqnarray*}
&~&\sup_{a\in\D}\int_{\D}
|f'(z)|^{p_2}(1-|z|^2)^{p_2-2}\left(1-|\sigma_a(z)|^2\right)^rdA(z)\\
&\leq& \|f\|_{\B}^{p_2-p_1} \sup_{a\in\D}\int_{\D}
|f'(z)|^{p_1}(1-|z|^2)^{p_1-2}\left(1-|\sigma_a(z)|^2\right)^sdA(z),
\end{eqnarray*}
which gives  $\Q_s^{p_1}(\D)\subseteq \Q_r^{p_2}(\D)$.

By \cite[Theorem 70]{ZZ},  if $p_1>p_2$, $ B_{p_1}(s)\subseteq B_{p_2}(r)$ if and only if
 $\f{1-s}{p_1}> \f{1-r}{p_2}$. Note that $f\in \Q_s^{p_1}(\D)$ if and only if
 $$
 \sup_{a\in \D}\|f\circ \sigma_a-f(a)\|_{B_{p_ 1}(s)}<\infty.
 $$
 Thus   $\Q_s^{p_1}(\D)\subseteq \Q_r^{p_2}(\D)$ for $p_1>p_2$ and $\f{1-s}{p_1}> \f{1-r}{p_2}$.

For  $s>r$, it is easy to construct a Blaschke  sequence $\{z_k\}$ satisfying
that $\sum_k (1-|z_k|)^s \delta_{z_k}$ is an $s$-Carleson measure and
$\sum_k (1-|z_k|)^r \delta_{z_k}$ is not an $r$-Carleson measure (see Lemma \ref{KC} where such a sequence is constructed with even more properties).
Applying Theorem \ref{T-Inner}, we get that the corresponding Blaschke product $B$ satisfies
 $$
 B \in \Q_s^{p_1}(\D)\setminus\Q_r^{p_2}(\D).
 $$
Let  $\f{1-s}{p_1}<\f{1-r}{p_2}$.  By \cite[Theorem 5.5]{Zha1},  the  lacunary  series
$$
g(z)=\sum_{k=0}^{\infty} 2^{-\f{k}{2} (\f{1-s}{p_1} + \f{1-r}{p_2})} z^{2^k } \in \Q_s^{p_1}(\D)\setminus\Q_r^{p_2}(\D).
$$
If $p_1>p_2$,  $s< r$ and $\f{1-s}{p_1}=\f{1-r}{p_2}$, the  lacunary series
$$
h(z)=\sum_{k=0}^{\infty} 2^{-k\f{1-s}{p_1}}k^{-\f{1}{p_2}} z^{2^k } \in \Q_s^{p_1}(\D)\setminus\Q_r^{p_2}(\D).
$$
Therefore, if $p_1\leq p_2$, then $\Q_s^{p_1}(\D)\subseteq \Q_r^{p_2}(\D)$ if and only if $s\leq r$.
If $p_1>p_2$, then $\Q_s^{p_1}(\D)\subseteq \Q_r^{p_2}(\D)$ if and only if $\f{1-s}{p_1}> \f{1-r}{p_2}$.
Hence it is enough  to prove that the inclusion relations between $\Q_s^{p_1}(\D)$ and $\Q_r^{p_2}(\D)$ are  the same as the inclusion relations between
$\Q_s^{p_1}(\T)$ and $\Q_r^{p_2}(\T)$.

Suppose $\Q_s^{p_1}(\D)\subseteq \Q_r^{p_2}(\D)$.
 Let $g\in \Q_s^{p_1}(\T)$. Without loss of generality we may assume that $g$ is real valued.  Denote by $\widetilde{g}$ the harmonic conjugate function of $\widehat{g}$. Set $h=\widehat{g}+i\widetilde{g}$. The Cauchy-Riemann equations give $|\nabla \widehat{g}(z)|\thickapprox |h'(z)|$. Then Theorem \ref{t3} shows that $|h'(z)|^{p_1}(1-|z|^2)^{p_1-2+s}dA(z)$ is an $s$-Carleson measure. Thus $h\in \Q_s^{p_1}(\D)$. So $h$ is also in $\Q_r^{p_2}(\D)$. Then
$|\nabla \widehat{g}(z)|^{p_2}(1-|z|^2)^{p_2-2+r}dA(z)$ is an $r$-Carleson measure, that is $g\in \Q_r^{p_2}(\T)$. Hence $\Q_s^{p_1}(\T)\subseteq \Q_r^{p_2}(\T)$.
On the other hand,  if  $\Q_s^{p_1}(\T)\subseteq \Q_r^{p_2}(\T)$, then Corollary \ref{C1} shows $\Q_s^{p_1}(\D)\subseteq \Q_r^{p_2}(\D)$. The proof is complete.
\end{proof}

\vskip  3mm
\noindent{\bf {\sl Remark 2}.}\ \ If $1<p_1<p_2<\infty$, $0<r<s<1$ and $\f{1-s}{p_1}\geq \f{1-r}{p_2}$, it is easy to check that  any  lacunary series in $\Q_s^{p_1}(\D)$ must be in $\Q_r^{p_2}(\D)$, but $\Q_s^{p_1}(\D)\nsubseteq \Q_r^{p_2}(\D)$. Thus we are in need to use inner functions to determine the inclusion relation in the proof of Theorem \ref{inclusions}.

\subsection{Logarithmic Carleson type measures}

Let $\alpha \geq 0$ and $s>0$.   A positive Borel measure $\mu$ on $\D$ is called an $\alpha$-logarithmic $s$-Carleson measure if
$$
 \sup_{I \subseteq \T}\f{1}{|I|^s} \left(\log\f{2}{|I|}\right)^\alpha  \mu(S(I))<\infty.
$$
 By \cite{Zha2}, $\mu$ is an $\alpha$-logarithmic $s$-Carleson measure if and only if
$$
\sup_{a\in \D}\left(\log\f{2}{1-|a|^2}\right)^\alpha\int_{\D}\left(\f{1-|a|^2}{|1-\bar{a}z|^2}\right)^sd\mu(z)<\infty.
$$
Condition \eqref{Eq1-2} in Theorem \ref{mt1} can be described in terms of  $\alpha$-logarithmic $s$-Carleson measure as follows.

\begin{lemma} \label{Lem2}
Let $1<p<\infty$ and $0<r<1$. Then the following conditions are equivalent.
\begin{enumerate}
\item [(1)]
$$
\sup_{I \subseteq \T} \f{1}{|I|^r} \left(\log\f{2}{|I|}\right)^p \int_I\int_I \f{|f(\zeta)-f(\eta)|^p}{|\zeta-\eta|^{2-r}}|d\zeta||d\eta| <\infty.
$$
\item[(2)] $|\nabla\widehat{f}(z)|^p(1-|z|^2)^{p-2+r}dA(z)$  is a $p$-logarithmic $r$-Carleson measure.
\end{enumerate}
\end{lemma}

\begin{proof} For an arc $I\subseteq \T$, by Lemma \ref{EL}, we have
$$
\int_I\int_I \f{|f(\zeta)-f(\eta)|^p}{|\zeta-\eta|^{2-r}}|d\zeta||d\eta|\lesssim \int_{S(3I)}|\nabla\widehat{f}(z)|^p(1-|z|)^{p-2+r}dA(z).
$$
Thus (2) implies (1).

Let (1) hold. Without loss of generality, let  $I$ be any arc centered at $1$. Then
$$
\int_I\int_I|f(e^{i\theta})-f(e^{it})|^pd\theta dt\lesssim \f{|I|^2}{\left(\log\f{2}{|I|}\right)^p}.
$$
Combining this with H\"older's inequality, we deduce that
\begin{eqnarray*}
\f{1}{|I|}\int_I|f(e^{i\theta})-f_I|d\theta
&\leq&\left(\f{1}{|I|}\int_I|f(e^{i\theta})-f_I|^pd\theta\right)^{1/p}\\
&\leq&\left(\f{1}{|I|^2}\int_I\int_I|f(e^{i\theta})-f(e^{it})|^pd\theta dt\right)^{1/p}\\
&\lesssim&\f{1}{\log\f{2}{|I|}}.
\end{eqnarray*}
Let $J=3I$. Using the above estimate and  a same argument in  \cite[p. 499]{Xi1}, we get
$$
\int_{|t|\geq \f{1}{3}|J|}|f(e^{it})-f_J|\f{dt}{t^2}\lesssim \f{1}{|I|\log\f{3}{|I|}}.
$$
Combining this with Lemma \ref{St-L}, we get that (2) is true.
\end{proof}

We also need the following result from \cite{PZ}.

\begin{otherl}\label{L-PZ}
 Let $p>1$ and $s> 0$. Let $\mu$ be a nonnegative Borel measure on $\D$. If $\mu$ is a $p$-logarithmic $s$-Carleson measure, then
 $$
 \int_{\D}|f(z)|^p d\mu(z)\lesssim \|f\|_{B_p(s)}^p
 $$
 for all $f\in B_p(s)$.
\end{otherl}

Now we are ready for the proofs of the main results of the paper.

\section{Proof of Theorem \ref{mt1}}

\subsection{Proof of part (1)}

Assume first  that $f\in M(\Q_s^{p_1}(\T), \Q_r^{p_2}(\T))$.
Set $g_w(z)=\log \f{2}{1-\bar w z}$, $w\in \D$. From \cite[Lemma 2.6]{PZ},
$$
\sup_{w\in \D}\|g_w\|_{\Q_s^{p_1}(\D)}<\infty,
$$
for all  $1<p_1<\infty$ and $0<s<1$.
This together with Corollary \ref{C1} shows  that $h_w(e^{i\theta})=\log \f{2}{1-\bar w e^{i\theta}}$ belongs to $\Q_s^{p_1}(\T)$ uniformly for  $w\in \D$, and hence the same is true for $g_ w=\Re h_ w$.
For any arc $I\subseteq \T$ centered at $e^{it}$ with $|I|<1/3$, take $a=(1-|I|)e^{it}$. Then $g_a(e^{i\theta})\thickapprox \log \f{2}{|I|}$ for all $e^{i\theta} \in I$. Since $f$ is a pointwise multiplier, then $f g_a\in \Q_r^{p_2}(\T)$  and it follows by Lemma \ref{L1} that  $fg_a\in BMO(\T)$.
By \cite[Lemma 2.6]{Steg1},
$$
\left|\f{1}{|I|}\int_I f(\zeta)g_a(\zeta)|d\zeta|\right|\lesssim \f{1}{\log |I|} \|fg_a\|_{BMO(\T)}\lesssim \f{1}{\log |I|}\|fg_a\|_{\Q_r^{p_2}(\T)}.
$$
Thus
$$
\left|\f{1}{|I|}\int_I f(\zeta)|d\zeta|\right|\lesssim \sup_{a\in \D}\|g_a\|_{\Q_s^{p_1}(\D)}<\infty,
$$
which shows $f\in L^\infty(\T)$.

Since
$$
g_a(e^{i\theta})(f(e^{it})-f(e^{i\theta}))=g_a(e^{it})f(e^{it})-g_a(e^{i\theta})f(e^{i\theta})+f(e^{it})\big (g_a(e^{i\theta})-g_a(e^{it})\big ),
$$
we get
\begin{equation*}
\f{1}{|I|^r}\int_I\int_I\f{|g_a(e^{i\theta})(f(e^{it})-f(e^{i\theta}))|^{p_2}}{|e^{i\theta}-e^{it}|^{2-r}}d\theta\,  dt
\lesssim \|g_a f\|_{\Q_r^{p_2}(\T)}^{p_2}+\|f\|_{L^\infty(\T)}^{p_2}\|g_a\|_{\Q_s^{p_1}(\T)}^{p_2}.
\end{equation*}
Note that $g_a(e^{i\theta})\thickapprox \log \f{2}{|I|}$ for all $e^{i\theta} \in I$. Thus
$$
\int_I\int_I\f{|f(e^{it})-f(e^{i\theta})|^{p_2}}{|e^{i\theta}-e^{it}|^{2-r}}d\theta dt\lesssim \f{|I|^r}{\left(\log\f{2}{|I|}\right)^{p_2}},
$$
which gives (1.2).\\

Next, suppose that $f\in L^\infty(\T)$ and (1.2) holds. We need to show  $f\in M(\Q_s^{p_1}(\T), \Q_r^{p_2}(\T))$. The proof of this implication is based on a technique developed in \cite{PP1} (see also \cite{PZ}).  By Lemma \ref{Lem2},
$|\nabla \widehat f(z)|^{p_2}(1-|z|^2)^{p_2-2+r}dA(z)$ is a $p_2$-logarithmic $r$-Carleson measure. Thus
 $$
 \sup_{a\in \D}\left(\log\f{2}{1-|a|}\right)^{p_2}\int_{\D}|\nabla\widehat f(z)|^{p_2} (1-|z|^2)^{p_2-2}(1-|\sigma_a(z)|^2)^rdA(z)<\infty.
 $$
 For all $g\in \Q_s^{p_1}(\T)$, we need to prove $gf\in  \Q_r^{p_2}(\T)$. Since $\widehat{f}\cdot \widehat{g}$ is an extension of $gf$, by Lemma \ref{EL}, it is enough to prove that
\[
I(a):=\int_{\D}\big |\nabla (\widehat f \, \widehat g)(z)\big |^{p_2} (1-|z|^2)^{p_2-2}(1-|\sigma_a(z)|^2)^rdA(z)\le C
\]
for some positive constant $C$ not depending on the point $a\in \D$. Using Theorem \ref{t3}, we have
\begin{eqnarray*}
I(a)&\lesssim&\int_{\D}|\widehat f(z)|^{p_2}|\nabla\widehat g(z)|^{p_2} (1-|z|^2)^{p_2-2}(1-|\sigma_a(z)|^2)^rdA(z)\\
&~&+\int_{\D}|\nabla\widehat f(z)|^{p_2}|\widehat g(z)|^{p_2} (1-|z|^2)^{p_2-2}(1-|\sigma_a(z)|^2)^rdA(z)\\
&\lesssim&\|f\|_{L^\infty(\T)}^{p_2}\cdot \|g\|_{\Q_r^{p_2}(\T)}^{p_2}+\int_{\D}|\nabla\widehat f(z)|^{p_2}|\widehat g(z)|^{p_2} (1-|z|^2)^{p_2-2}(1-|\sigma_a(z)|^2)^rdA(z).
\end{eqnarray*}
If $p_1\leq p_2$ and $s\leq r$, Theorem  \ref{inclusions} gives $\Q_s^{p_1}(\T) \subseteq \Q_r^{p_2}(\T)$, and by  the closed-graph theorem,
$
\|g\|_{\Q_r^{p_2}(\T)}\lesssim \|g\|_{\Q_s^{p_1}(\T)}
$
for all $g\in \Q_s^{p_1}(\T)$. Hence, we have
\[
I(a)\lesssim \|f\|_{L^\infty(\T)}^{p_2}\cdot \|g\|_{\Q_s^{p_1}(\T)}^{p_2}+\int_{\D}|\widehat g(z)|^{p_2} \,|\nabla\widehat f(z)|^{p_2} (1-|z|^2)^{p_2-2}(1-|\sigma_a(z)|^2)^rdA(z).
\]
Without loss of generality, we may assume that $g$ is real valued.  Let  $\widetilde{g}$ be the harmonic conjugate function of $\widehat{g}$. Set $h=\widehat{g}+i\widetilde{g}$. The Cauchy-Riemann equations give $|\nabla \widehat{g}(z)|\thickapprox |h'(z)|$. Then  $h\in \Q_s^{p_1}(\D)\subseteq \Q_r^{p_2}(\D)$ and $|\widehat g(z)|\leq |h(z)|$. Hence
\[
\begin{split}
\int_{\D}|\widehat g(z)|^{p_2}&\,|\nabla\widehat f(z)|^{p_2} \,(1-|z|^2)^{p_2-2}\,(1-|\sigma_a(z)|^2)^r \,dA(z)
\\
&\leq\int_{\D}|h(z)|^{p_2}\,|\nabla\widehat f(z)|^{p_2}\, (1-|z|^2)^{p_2-2}\,(1-|\sigma_a(z)|^2)^r \,dA(z)
\\
&\lesssim \int_{\D}|h(a)|^{p_2}\,|\nabla\widehat f(z)|^{p_2}\, (1-|z|^2)^{p_2-2}\,(1-|\sigma_a(z)|^2)^r \,dA(z)
\\
&\quad +\int_{\D}|h(z)-h(a)|^{p_2}\,|\nabla\widehat f(z)|^{p_2}\, (1-|z|^2)^{p_2-2}\,(1-|\sigma_a(z)|^2)^r \,dA(z)
\\
&\thickapprox T_1+T_2.
\end{split}
\]
Since  $\Q_s^{p_1}(\D)$ is a subspace of the Bloch space $\B$, then any function $h\in \Q_s^{p_1}(\D)$ has the following growth:
$$
|h(z)|\lesssim \|h\|_{\B}\log\f{2}{1-|z|}\lesssim \|h\|_{\Q_s^{p_1}(\D)}\log\f{2}{1-|z|}
$$
for all $z\in \D$. Thus
\begin{eqnarray*}
T_1\lesssim \sup_{a\in \D} \left(\log\f{2}{1-|a|}\right)^{p_2}\int_{\D}|\nabla\widehat f(z)|^{p_2} (1-|z|^2)^{p_2-2}(1-|\sigma_a(z)|^2)^rdA(z)<\infty.
\end{eqnarray*}
Applying  Lemma \ref{L-PZ}, we see  that
\begin{eqnarray*}
T_2&=&(1-|a|^2)^r\int_{\D}\left|\f{h(z)-h(a)}{(1-\bar{a}z)^{\f{2r}{p_2}}}\right|^{p_2}|\nabla\widehat f(z)|^{p_2} (1-|z|^2)^{p_2-2+r}dA(z)\\
&\lesssim&(1-|a|^2)^r\left(|h(0)-h(a)|^{p_2}+\int_{\D}\left|\left(\f{h(z)-h(a)}{(1-\bar{a}z)^{\f{2r}{p_2}}}\right)'\right|^{p_2}(1-|z|^2)^{p_2-2+r}dA(z)\right)\\
&\lesssim&(1-|a|^2)^r\|h\|_{\mathcal{B}}^{p_2}\left(\log\f{2}{1-|a|}\right)^{p_2}+\int_{\D}|h'(z)|^{p_2}(1-|z|^2)^{p_2-2}(1-|\sigma_a(z)|^2)^rdA(z)\\
&~&+\int_{\D}\f{|h(z)-h(a)|^{p_2}}{|1-\bar{a}z|^{p_2}}(1-|z|^2)^{p_2-2}(1-|\sigma_a(z)|^2)^rdA(z)\\
&\lesssim&\|h\|_{\Q_s^{p_1}(\D)}^{p_2}+\int_{\D}\f{|h(z)-h(a)|^{p_2}}{|1-\bar{a}z|^{p_2}}(1-|z|^2)^{p_2-2}(1-|\sigma_a(z)|^2)^rdA(z).
\end{eqnarray*}
By \cite[Proposition 2.8]{PZ}, one gets
$$
\int_{\D}\f{|h(z)-h(a)|^{p_2}}{|1-\bar{a}z|^{p_2}}(1-|z|^2)^{p_2-2}(1-|\sigma_a(z)|^2)^rdA(z)\lesssim \|h\|_{\Q_r^{p_2}(\D)}^{p_2}.
$$
Combining the above estimates, we obtain
$$
\sup_{a\in \D}I(a)<\infty.
$$
Thus $f\in M(\Q_s^{p_1}(\T), \Q_r^{p_2}(\T))$.

\subsection{Proof of part (2)}

Let  $p_1>p_2$ and $s\leq r$. If $\f{1-s}{p_1}>\f{1-r}{p_2}$, Theorem  \ref{inclusions} gives  the inclusion $\Q_s^{p_1}(\T) \subseteq \Q_r^{p_2}(\T)$. Hence
$$
M(\Q_r^{p_2}(\T), \Q_r^{p_2}(\T))\subseteq M(\Q_s^{p_1}(\T), \Q_r^{p_2}(\T)).
$$
Checking the proof in (1), one gets that $f\in M(\Q_s^{p_1}(\T), \Q_r^{p_2}(\T))$ if and only if $f\in L^\infty(\T)$ and $f$ satisfies (1.2).\\

In case that $\f{1-s}{p_1}\leq \f{1-r}{p_2}$, as has been observed in the proof of Theorem \ref{inclusions}, there is a  lacunary Fourier series
in $\Q_s^{p_1}(\T)\setminus \Q_r^{p_2}(\T)$. Then we get
 $ M(\Q_s^{p_1}(\T), \Q_r^{p_2}(\T))=\{0\}$ as a consequence of the following result.

\begin{lemma} \label{LS} Let $1<p_1,\ p_2<\infty$ and $0<s,\ r<1$. If there exists a lacunary Fourier series
$
g(e^{i\theta})=\sum_{k=0}^{\infty}a_k e^{i2^k \theta} \in \Q_s^{p_1}(\T)\setminus \Q_r^{p_2}(\T),
$
 then $ M(\Q_s^{p_1}(\T),\ \Q_r^{p_2}(\T))=\{0\}$.
\end{lemma}

\begin{proof}
We adapt an argument from \cite{GGP}. By Corollary \ref{C1} and the lacunary series characterization of $\Q_s^{p_1}(\D)$ spaces in  \cite[Theorem 5.5]{Zha1}, we see that
$$
\|g\|_{\Q_s^{p_1}(\T)}^{p_1} \thickapprox \sum_{k=0}^\infty |a_k|^{p_1}\,2^{k(1-s)}<\infty
$$
and
$$
\sum_{k=0}^\infty |a_k|^{p_2} \,2^{k(1-r)}=\infty.
$$
Let $\{r_k(t)\}_{k=0}^\infty$ be the sequence of Rademacher functions and consider the function
$$
g_t(z)=\sum_{k=0}^\infty r_k(t)a_kz^{2^k},\ \ 0\leq t\leq 1.
$$
Then $g_t\in\Q_s^{p_1}(\T)$ with $\|g_t\|_{\Q_s^{p_1}(\T)}\thickapprox\|g\|_{\Q_s^{p_1}(\T)}$. Suppose $f\in M(\Q_s^{p_1}(\T), \Q_r^{p_2}(\T))$.
For any $a\in\D$, we have
\[
\begin{split}
\int_{0}^{1} I_{g_ t}(f,a)\,dt&:=\int_0^1\int_\T\int_\T\f{|f(\zeta)(g_t(\zeta)-g_t(\eta))|^{p_2}}{|\zeta-\eta|^{2-r}}\left(\f{1-|a|^2}{|\zeta-a||\eta-a|}\right)^r|d\zeta||d\eta|dt\\
\\
&\lesssim \int_0^1\int_\T\int_\T\f{|f(\zeta)g_t(\zeta)-f(\eta)g_t(\eta)|^{p_2}}{|\zeta-\eta|^{2-r}}\left(\f{1-|a|^2}{|\zeta-a||\eta-a|}\right)^r|d\zeta||d\eta|dt\\
\\
& \quad  +\int_0^1\int_\T\int_\T\f{|(f(\zeta)-f(\eta))g_t(\eta)|^{p_2}}{|\zeta-\eta|^{2-r}}\left(\f{1-|a|^2}{|\zeta-a||\eta-a|}\right)^r|d\zeta||d\eta|dt.
\end{split}
\]
Hence, using Fubini's Theorem,   we obtain
\[
\begin{split}
\int_{0}^{1} \!\! I_{g_ t}(f,a)\,dt&\lesssim \int_{0}^{1}\|fg_t\|_{\Q_r^{p_2}(\T)}^{p_2}\,dt\\
\\
&\quad +\int_\T\int_\T\f{|f(\zeta)-f(\eta)|^{p_2}}{|\zeta-\eta|^{2-r}} \left(\int_0^1|g_t(\eta)|^{p_2}dt\right) \left(\f{1-|a|^2}{|\zeta-a||\eta-a|}\right)^r|d\zeta||d\eta|.
\end{split}
\]
Since $f\in M(\Q_s^{p_1}(\T), \Q_r^{p_2}(\T))$, then $\|fg_t\|_{\Q_r^{p_2}(\T)}\lesssim \|g\|_{\Q_s^{p_1}(\T)}$. Also, by Khinchine's inequality
\[
\int_0^1|g_t(\eta)|^{p_2}\,dt \thickapprox \left(\sum_{k=0}^\infty |a_k|^2\right)^{\f{{p_2}}{2}} \thickapprox \|g\|_{H^2}^{p_ 2}\lesssim \|g\|_{\Q_s^{p_1}(\T)},
\]
because $\Q_s^{p_1}(\D) \subseteq H^2$. Combining these estimates with Theorem \ref{t3} we have

\begin{equation}\label{Eq-Ig1}
\int_{0}^{1} \!\! I_{g_ t}(f,a)\,dt
\lesssim \|g\|_{\Q_s^{p_1}(\T)}^{p_2}+\|g\|_{\Q_s^{p_1}(\T)}^{p_2}\cdot \|f\|_{\Q_r^{p_2}(\T)}^{p_2}<\infty,
\end{equation}
because, as $1\in \Q_s^{p_1}(\T)$, then $f=f\cdot 1 \in \Q_r^{p_2}(\T)$.\\

Now, if $f\neq 0$, then there exists a positive constant $C$ such that
\begin{equation}\label{Eqfzero}
\int_0^{2\pi}\!|f(e^{i\theta})|d\theta\geq C.
\end{equation}
Fubini's theorem and a change of variables give
\[
\int_0^1 \!\! I_{g_ t}(f,0)\,dt\thickapprox \int_0^{2\pi}\!\! \f{dh}{h^{2-r}}\int_0^{2\pi}\!\!|f(e^{i\theta})|d\theta\int_0^1|g_t(e^{i\theta})-g_t(e^{i(\theta+h)})|^{p_2}dt.
\]
Applying  Khinchine's inequality again, we see  that
\[
\int_0^1|g_t(e^{i\theta})-g_t(e^{i(\theta+h)})|^{p_2}dt  \thickapprox \left(\sum_{k=0}^\infty |a_k|^2|1-e^{i2^kh}|^2\right)^{\f{{p_2}}{2}}.
\]
Hence, by \eqref{Eqfzero},
\begin{eqnarray*}
\int_0^1 \!\! I_{g_ t}(f,0)\,dt&\thickapprox&\int_0^{2\pi}\f{\left(\sum_{k=0}^\infty |a_k|^2|1-e^{i2^kh}|^2\right)^{\f{{p_2}}{2}}}{h^{2-r}}dh\int_0^{2\pi}|f(e^{i\theta})|d\theta\\
&\gtrsim
&\int_0^{2\pi}\f{dh}{h^{2-r}}\int_0^1|g_t(e^{i\theta})-g_t(e^{i(\theta+h)})|^{p_2}dt,
\end{eqnarray*}
for any $\theta\in[0,\ 2\pi)$. Thus, using  \eqref{Eq-C1} and the  Khinchine inequality, we have
\begin{eqnarray*}
\int_0^1 \!\! I_{g_ t}(f,0)\,dt
&\gtrsim&\int_0^{2\pi}\f{dh}{h^{2-r}}\int_0^{2\pi} \int_0^1|g_t(e^{i\theta})-g_t(e^{i(\theta+h)})|^{p_2}dtd\theta\\
&\thickapprox&\int_0^1\int_\T\int_\T\f{|g_t(\zeta)-g_t(\eta)|^{p_2}}{|\zeta-\eta|^{2-r}}|d\zeta||d\eta|dt\\
&\thickapprox&\int_0^1\int_\D|g_t'(z)|^{p_2}(1-|z|^2)^{p_2-2+r}dA(z)dt\\
&\thickapprox&\int_\D\left(\int_0^1|g_t'(z)|^{p_2}dt\right)(1-|z|^2)^{p_2-2+r}dA(z)\\
&\thickapprox&\int_\D\left(M_2(|z|, g')\right)^{p_2}(1-|z|^2)^{p_2-2+r}dA(z).
\end{eqnarray*}
Since $g'(z)$ is also a lacunary series, it is well known that $M_2(|z|, g')\thickapprox M_{p_2}(|z|, g')$ (see \cite{Zy} for example). Hence
\begin{eqnarray*}
\int_0^1 \!\! I_{g_ t}(f,0)\,dt
&\gtrsim&\int_\D\left(M_{p_2}(|z|, g')\right)^{p_2}(1-|z|^2)^{p_2-2+r}dA(z)\\
&\thickapprox&\int_\D|g'(z)|^{p_2}(1-|z|^2)^{p_2-2+r}dA(z)\\
&\thickapprox& \sum_{k=0}^\infty |a_k|^{p_2}2^{k(1-r)}=\infty.
\end{eqnarray*}
This  contradicts \eqref{Eq-Ig1}. Thus $ M(\Q_s^{p_1}(\T), \Q_r^{p_2}(\T))=\{0\}$.
\end{proof}

\subsection{Preliminaries for the proof of part (3)}

Lemma \ref{LS} shows that $M(\Q_s^{p_1}(\T), \Q_r^{p_2}(\T))$ is trivial for a wide range of values of the parameters $p_1$, $s$, $p_2$, $r$. However Lemma \ref{LS} miss the case that $p_1<p_2$, $s>r$ and $\f{1-s}{p_1}\geq \f{1-r}{p_2}$. In this case, $\Q_s^{p_1}(\T)\nsubseteq \Q_r^{p_2}(\T)$, but  any  lacunary Fourier series in $\Q_s^{p_1}(\T)$ must be in $\Q_r^{p_2}(\T)$. Thus,  we are in need to look for another method to determine that $M(\Q_s^{p_1}(\T), \Q_r^{p_2}(\T))$ is trivial in this case. We are going to use the tangential boundary approximation results of Nagel-Rudin-Shapiro \cite{NRS} in order to handle this case.

Given $\zeta \in \T$ and $\alpha >1$,  let
$$
\Gamma_{\alpha}(\zeta)=\{z\in \D: |1-\overline{\zeta}z|<\alpha (1-|z|)\}
$$
be a Stolz angle with vertex at $\zeta$.
If $f\in  H^p$,  then its non-tangential limit  exists almost everywhere. Namely, for almost every $\zeta \in \T$, the limit
$$
f(\zeta):=\lim_{\stackrel{z\in \Gamma_{\alpha}(\zeta)}{z\rightarrow \zeta}} f(z)
$$
exists. In order to handle part (3), for every $\zeta \in \T$, we need to construct a Blaschke sequence $\{a_ k\}$ converging to $\zeta$ in a way that the associated Blaschke product is in $\Q ^{p_ 1}_ s(\D)\setminus \Q^{p_ 2}_ r(\D)$. In view of Theorem \ref{T-Inner},   $\sum_k (1-|a_k|)^s \delta_{a_k}$ must be an $s$-Carleson measure, but $\sum_k (1-|a_k|)^r \delta_{a_k}$ can not be  an $r$-Carleson measure. According to \cite{GPV} this is not possible if the sequence $\{a_ k\}$ converges to $\zeta$ non-tangentially. \\

For $c>0$ and $\delta>1$, consider the region
$$
\Omega_{\delta, c}(\theta)=\left\{re^{i\varphi}\in \D: 1-r>c \left|\sin\f{\varphi-\theta}{2}\right|^\delta\right\}.
$$
 Then $\Omega_{\delta, c}(\theta)$ touches $\T$ at $e^{i\theta}$ tangentially. We say that a function $h$, defined in $\D$, has
 $\Omega_\delta$-limit $L$ at $e^{i\theta}$ if $h(z)\rightarrow L$ as $z\rightarrow e^{i\theta}$ within $\Omega_{\delta, c}(\theta)$ for every $c$.
 A. Nagel, W. Rudin and J. Shapiro \cite{NRS} obtained
 the following result.

\begin{otherth}\label{T-NRS} Suppose $1\leq p<\infty$, $f\in L^p(\T)$, $0<\alpha<1$, and
$$
h(z)=\f{1}{2\pi}\int_{-\pi}^{\pi}\f{f(e^{i\theta})d\theta}{(1-e^{-i\theta}z)^{1-\alpha}}, \ \ z\in \D.
$$
If $\alpha p<1$ and $\delta=1/(1-\alpha p)$, then the  $\Omega_\delta$-limit of $h$ exists almost everywhere on $\T$.
\end{otherth}

For $\beta \in \R$ and $0<p<\infty$, the Hardy-Sobolev space $H^p_\beta$ consists of analytic functions $f$ in $\D$ such that $D^\beta f\in H^p$, where
$f(z)=\sum_{k=0}^\infty a_k z^k$ is the Taylor expansion of $f$ and
$$
D^\beta f(z)=\sum_{k=0}^\infty (1+k)^\beta a_k z^k.
$$

\begin{prop} \label{P-TA}
Let $1<p<\infty$ and $0<s<t<1$. Suppose $h\in B_p(s)$.
Then the  $\Omega_{1/t}$-limit of $h$ exists almost everywhere on $\T$.
\end{prop}

\begin{proof} By Theorem \ref{T-NRS}, it is enough to show that there exists $f\in L^p(\T)$ with
$$
h(z)=\f{1}{2\pi}\int_{-\pi}^{\pi}\f{f(e^{i\theta})d\theta}{(1-e^{-i\theta}z)^{1-\f{1-t}{p}}}, \qquad z\in \D.
$$
Thus we only need to prove that $h\in H^p_{\f{1-t}{p}}$. Note that \cite[Theorem 2.19]{ZhuBn}
$$
\int_{\D} \left|D^{1+\f{1-t}{p}}h(z)\right|^p (1-|z|)^{p-1+s-t}dA(z)
\thickapprox \int_{\D} \left|h'(z)\right|^p (1-|z|)^{p-2+s}dA(z)<\infty.
$$
Thus $D^{\f{1-t}{p}}h \in B_p(1+s-t)$. Since $s<t$, then $1+s-t<1$ and therefore  $B_p(1+s-t)\subseteq H^p$ (see \cite[Lemma 2.4]{ABP} for example). Hence we get  $h\in H^p_{\f{1-t}{p}}$. The proof is complete.
\end{proof}

In case that $p\le 2$, it is known \cite{GP} that one can take $t=s$ in Proposition \ref{P-TA}. The following construction will be a key for the proof of part (3).

\begin{lemma} \label{KC} Let $0<r<s<1$ and $0<t<1$.  For every $e^{i\theta}\in \T$, there exists a Blaschke  sequence $\{a_k\}_{k=1}^\infty$ satisfying the following conditions.
\begin{enumerate}
\item [(a)] $e^{i\theta}$ is the unique accumulation point of $\{a_k\}$.
\item [(b)] $\{a_k\}\subseteq \Omega_{1/t, c}(\theta)$ for some $c>0$.
\item [(c)] $\sum_k (1-|a_k|)^s \delta_{a_k}$ is an $s$-Carleson measure.
\item [(d)] $\sum_k (1-|a_k|)^r \delta_{a_k}$ is not an $r$-Carleson measure.
\end{enumerate}
\end{lemma}

\begin{proof}Set
$$
a_k=\left(1-k^{-\f{1}{\varepsilon}}\right)e^{i\theta_k}, \ \ k=1, 2, \cdots,
$$
where $\varepsilon>0$ is taken so that
$$
r(1-t)<\varepsilon< \min \{r, s(1-t)\}
$$
and
$$
\theta_k=k^{-\f{t}{\varepsilon}}+\theta.
$$
Clearly, $\{a_ k\}$ is a Blaschke sequence and $e^{i\theta}$ is its unique accumulation point with $\{a_k\}\subseteq \Omega_{1/t, c}(\theta)$ for some $c>0$.
 To prove that $\sum_k (1-|a_k|)^s \delta_{a_k}$ is an $s$-Carleson measure,
it is enough to consider sufficiently small arcs  $I\subseteq \T$ centered at $e^{i\theta}$. Since  $\varepsilon<  s(1-t)<s$, we deduce that
\begin{eqnarray*}
\sum_{a_k\in S(I)}(1-|a_k|)^s
&\thickapprox&\sum_{|\theta_k-\theta|\leq\f{|I|}{2}}k^{-\f{s}{\varepsilon}}
\thickapprox\sum_{k\geq \left(\f{2}{|I|}\right)^{\f{\varepsilon}{t}}}k^{-\f{s}{\varepsilon}}\\
&\thickapprox&\int_{\left(\f{2}{|I|}\right)^{\f{\varepsilon}{t}}}^\infty x^{-\f{s}{\varepsilon}}dx
\thickapprox |I|^{\f{s-\varepsilon}{t}}
\lesssim |I|^s,
\end{eqnarray*}
which gives that  $\sum_k (1-|a_k|)^s \delta_{a_k}$ is an $s$-Carleson measure. On the other hand,
for $r(1-t)<\varepsilon< r$, one gets
$$
\sum_{a_k\in S(I)}\f{(1-|a_k|)^r}{|I|^r}\thickapprox \f{1}{|I|^r}\int_{\left(\f{2}{|I|}\right)^{\f{\varepsilon}{t}}}^\infty x^{-\f{r}{\varepsilon}}dx
\thickapprox|I|^{\f{r-\varepsilon}{t}-r}\rightarrow \infty, \ \text{as} \  |I|\rightarrow 0.
$$
Thus $\sum_k (1-|a_k|)^r \delta_{a_k}$ is not an $r$-Carleson measure. The proof is complete.
\end{proof}
\mbox{}
\\
We also need the following estimate.

\begin{lemma} \label{LIn} Let $1<q<\infty$ and $0<r<1$. Let $f\in L^q(\T)$ and $S$  an inner function.  For $a\in \D$,
\begin{eqnarray*}
 J(a)&:=& \int_\D \widehat{|f|^q}(z)\,(1-|S(z)|)^q \,(1-|z|)^{-2} \left(1-|\sigma_a(z)|\right)^r \,  dA(z)\\
&\lesssim&
\int_\T\int_\T\f{|f(\zeta)|^q \, |S(\zeta)-S(\eta))|^q}{|\zeta-\eta|^{2-r}}\left(\f{1-|a|^2}{|\zeta-a||\eta-a|}\right)^r\, |d\zeta|\,|d\eta|.
\end{eqnarray*}
\end{lemma}
\begin{proof} We first consider the case  $a=0$. Using Fubini's Theorem, we see that
\begin{eqnarray*}
J(0)
&=& \int_{\T}|f(\zeta)|^q \left(\int_\D\f{(1-|S(z)|)^q}{|\zeta-z|^2}(1-|z|^2)^{r-1}dA(z)\right) |d\zeta|\\
&\leq& \int_{\T}|f(\zeta)|^q \left(\int_\D\f{|S(\zeta)-S(z)|^q}{|\zeta-z|^2}(1-|z|^2)^{r-1}dA(z)\right) |d\zeta|.
\end{eqnarray*}
Note that
\begin{eqnarray*}
 |S(\zeta)-S(z)|^q
&\leq& \int_{\T}\left|S(\zeta)-S(\eta)\right|^q \f{1-|z|^2}{|\eta-z|^2}|d\eta|.
\end{eqnarray*}
Consequently, by the estimate \cite{OF},
\[
\int_\D\f{(1-|z|^2)^r}{|\zeta-z|^2\, |\eta-z|^2}\, dA(z) \lesssim \frac{1}{|\zeta-\eta|^{2-r}},
\]
we have
\begin{eqnarray*}
J(0)
&\le & \int_{\T}\int_{\T}\left(\int_\D\f{(1-|z|^2)^r}{|\zeta-z|^2\, |\eta-z|^2}\, dA(z)\right)|f(\zeta)|^q \,|S(\zeta)-S(\eta)|^q \, |d\eta|\, |d\zeta|\\
&\lesssim& \int_{\T}\int_{\T}\f{|f(\zeta)|^q \, |S(\zeta)-S(\eta)|^q}{|\zeta-\eta|^{2-r}} \, |d\eta|\, |d\zeta|.
\end{eqnarray*}
That is,
$$
\int_\D \widehat{|f|^q}(z)(1-|S(z)|)^q (1-|z|^2)^{r-2}dA(z)\lesssim
\int_{\T}\int_{\T}\f{|f(\zeta)|^q \, |S(\zeta)-S(\eta)|^q}{|\zeta-\eta|^{2-r}}\,  |d\eta|\, |d\zeta|.
$$
Replacing $f$ and $S$ by $f\circ \sigma_a$ and $S\circ \sigma_a$ in the above inequality  respectively and changing the variables, the desired result follows.
\end{proof}

\subsection{Proof of part (3)}
If $s>r$, for any $e^{i\theta}\in \T$,  we take the sequence $\{a_k\}_{k=1}^\infty\subseteq \D$ constructed in Lemma \ref{KC}, and let $B$ be the corresponding Blaschke product. Then Theorem \ref{T-Inner}  shows that $B\in \Q_s^{p_1}(\T)$.
If $f\in  M(\Q_s^{p_1}(\T), \Q_r^{p_2}(\T))$ then $fB \in \Q_r^{p_2}(\T)$ and
\[
\begin{split}
\sup_{a\in \D}&\int_\T\int_\T\f{|f(\zeta)(B(\zeta)-B(\eta))|^{p_2}}{|\zeta-\eta|^{2-r}}\left(\f{1-|a|^2}{|\zeta-a||\eta-a|}\right)^r|d\zeta||d\eta|\\
\\
&\lesssim \| fB\|_{\Q_r^{p_2}(\T)}^{p_2}+\|f\|_{\Q_r^{p_2}(\T)}^{p_2}<\infty.
\end{split}
\]
Then Lemma \ref{LIn} gives
$$
\sup_{a\in \D}\int_\D \widehat{|f|^{p_2}}(z)(1-|B(z)|)^{p_2} (1-|z|)^{-2} \left(1-|\sigma_a(z)|\right)^r dA(z)<\infty.
$$
Since the sequence $\{a_ k\}$ is Carleson-Newman, that is
$$
\sup_{a\in\D}\sum_{k=1}^\infty \left(1-|\sigma_a(a_k)|^2\right)<\infty,
$$
we have (see \cite[p. 1292]{ABP} for example) that
$$
(1-|B(z)|)^{p_ 2}\gtrsim \left (\sum_{k=1}^\infty \big (1-|\sigma_{a_k}(z)|^2 \big ) \right )^{p_ 2}\ge \sum_{k=1}^\infty \left (1-|\sigma_{a_k}(z)|^2\right)^{p_ 2}, \qquad z\in \D.
$$
This gives
\begin{equation}\label{EqS4}
\sup_{a\in \D} \sum_{k=1}^\infty  \int_{E(a_k)}\widehat{|f|^{p_2}}(z) \left(1-|\sigma_{a_k}(z)|^2\right)^{p_2} (1-|z|)^{-2} \left(1-|\sigma_a(z)|\right)^r \, dA(z)<\infty,
\end{equation}
where
$
E(a_k)=\{w\in \D:  |\sigma_{a_k}(w)|<1/2\}
$
is a pseudo-hyperbolic disk centered at $a_ k$.
It is well known that
$$
(1-|z|^2)^2\thickapprox |1-\overline{a_k}z|^2\thickapprox (1-|a_k|^2)^2
$$
for all $z\in E(a_k)$. Furthermore, by \cite[Lemma 4.30]{Zhu2},
$$
1-|\sigma_a(z)|^2\thickapprox 1-|\sigma_a(a_k)|^2
$$
for all $a\in \D$ and $z\in E(a_k)$. This, \eqref{EqS4} and subharmonicity yield
\[
\sup_{a\in \D} \sum_{k=1}^\infty (\widehat{|f|}(a_k))^{p_2}(1-|\sigma_a(a_k)|)^r<\infty.
\]
Since
$$
\sup_{a\in \D} \sum_{k=1}^\infty (1-|\sigma_a(a_k)|)^r=\infty,
$$
this forces $\widehat{|f|}(a_k) \rightarrow 0$.
 Note that $|f|\in \Q_r^{p_2}(\T)$ and $\lim_{k\rightarrow \infty}a_k=e^{i\theta}$. Then $g=\widehat{|f|}+i\widetilde{|f|} \in \Q_r^{p_2}(\D)\subset B_{p_2}(r)$ because of Theorem \ref{t3}, since the Cauchy-Riemann equations give $|g'(z)|\thickapprox |\nabla (\widehat{|f|})(z)|$. Take $t>r$ in Lemma \ref{KC}, and apply Proposition \ref{P-TA}, to get
 $$
 (\widehat{|f|}+i\widetilde{|f|})(e^{i\theta})=\lim_{k\rightarrow \infty}(\widehat{|f|}+i\widetilde{|f|})(a_k)
 $$
 for almost every $e^{i\theta}\in \T$. This implies
  $$
 \widehat{|f|}(e^{i\theta})=\lim_{k\rightarrow \infty} \widehat{|f|}(a_k)=0 \ \ a.e. \ \ e^{i\theta}\in \T.
 $$
 If we take the normalized  $\widetilde{|f|}$ with $\widetilde{|f|}(0)=0$,
then the analytic function  $\widehat{|f|}+i\widetilde{|f|}$  vanishes on $\T$ almost everywhere. Hence $\widehat{|f|}+i\widetilde{|f|}$  vanishes on the whole disk. Then $\widehat{|f|}(z)\equiv 0$, $z\in \D$.
This implies $f(\zeta)=0$ for almost every $\zeta \in \T$. The proof is complete.

\section{Proof of Theorem \ref{t2}}
  Let $f$ be the symbol of a bounded multiplication operator $M_f$ on $\qpst$ space. If
$\lambda \not\in \sigma(M_f)$, then $M_f-\lambda E$ is invertible. Clearly,  the inverse operator of $M_f-\lambda E$ is $M_{\f{1}{f-\lambda}}$. By the open mapping theorem, $M_{\f{1}{f-\lambda}}$ is also bounded on $\qpst$, and  Theorem \ref{mt1} gives $\f{1}{f-\lambda} \in L^\infty(\T)$. Then
$$
\left|\f{1}{f(\zeta)-\lambda}\right|<2 \Big \|\f{1}{f-\lambda}\Big\|_{L^\infty(\T)}, \qquad a.e. \quad \zeta \in \T.
$$
Namely, the set
$$
\left\{\zeta \in \T: |f(\zeta)-\lambda|< \left ( 2\left\|\f{1}{f-\lambda}\right\|_{L^\infty(\T)}\right  )^{-1} \right\}
$$
has measure zero. Thus $\lambda \not\in \mathcal R(f)$.

Conversely,  let $\lambda \not\in \mathcal R(f)$. Then there exists some positive constant  $\delta$ such that the set
$\{\zeta \in \T: |f(\zeta)-\lambda|<\delta\} $ has measure zero. Hence
$$
\f{1}{f(\zeta)-\lambda}\leq \f{1}{\delta},  \ \ a.e. \ \ \zeta \in \T.
$$
Thus  $M_f-\lambda E$ is injective.
Using $f\in M(\qpst)$ and Theorem \ref{mt1}, we obtain
$$
\sup_{I\subseteq \T} \f{1}{|I|^s}\left(\log\f{2}{|I|}\right)^p\int_I\int_I \f{\left|\f{1}{f(\zeta)-\lambda}-\f{1}{f(\eta)-\lambda}\right|^p}{|\zeta-\eta|^{2-s}}|d\zeta||d\eta| <\infty.
$$
Applying Theorem \ref{mt1} again, one gets $\f{1}{f-\lambda}\in M(\qpst)$. Then for any $g\in \qpst$, we obtain $\f{g}{f-\lambda}\in \qpst$ and
$$
(M_f-\lambda E)\f{g}{f-\lambda}=g.
$$
Then   $M_f-\lambda E$ is surjective. Thus $M_f-\lambda E$ is invertible and hence $\lambda \not\in \sigma(M_f)$.

\section{The analytic version of Theorems \ref{mt1} and \ref{t2}}

\begin{thm}\label{AV-1} Let $1<p_1, p_2<\infty$ and $0<s, r<1$. Then the following are true.
\begin{enumerate}
\item [(1)] If $p_1\leq p_2$ and $s\leq r$, then $f\in M(\Q_s^{p_1}(\D), \Q_r^{p_2}(\D))$ if and only if $f\in H^\infty$ and
\begin{equation}\label{Eq5-1}
\sup_{I\subseteq \D} \f{1}{|I|^r}\left(\log\f{2}{|I|}\right)^{p_2}\int_I\int_I \f{|f(\zeta)-f(\eta)|^{p_2}}{|\zeta-\eta|^{2-r}}|d\zeta||d\eta| <\infty.
\end{equation}
\item[(2)]Let  $p_1>p_2$ and $s\leq r$. If $\f{1-s}{p_1}>\f{1-r}{p_2}$, then $f\in M(\Q_s^{p_1}(\D), \Q_r^{p_2}(\D))$ if and only if $f\in H^\infty$ and $f$ satisfies \eqref{Eq5-1}. If $\f{1-s}{p_1}\leq\f{1-r}{p_2}$, then  $M(\Q_s^{p_1}(\D), \Q_r^{p_2}(\D))=\{0\}$.
\item[(3)] If $s>r$, then $M(\Q_s^{p_1}(\D), \Q_r^{p_2}(\D))=\{0\}$.
\end{enumerate}
\end{thm}

\begin{proof} Just  follow the proof of Theorem \ref{mt1}. Now we give a different proof of  (3) using zero sets. Set
$$
z_n=\left(1-\f{1}{n^{1/t}}\right)e^{i\theta_n},\ \ n=2, 3, \cdots,
$$
where $r<t<s$ and
$$
\theta_n=\sum_{k=1}^{n-1}\f{1}{k}+\f{1}{2n},\ \ n=2, 3, \cdots.
$$
By \cite[Theorem 8]{PP2}, $\sum_n (1-|z_n|)^s \delta_{z_n}$ is an $s$-Carleson measure. Also, since $\sum_{n=1}^\infty (1-|z_n|)^t=\infty$, it follows from the proof of \cite[Theorem 5.10]{NRS} that $\{z_n\}$ is not a zero set of $B_{p_ 2}(r)$ (look also at the proof of Proposition \ref{P-TA}).
Let $B$ be the Blaschke product with zero sequence $\{z_n\}$. Then $B\in \Q_s^{p_1}(\D) \setminus \Q_r^{p_2}(\D)$. If
$f\in M(\Q_s^{p_1}(\D), \Q_r^{p_2}(\D))$, then  $fB\in \Q_r^{p_2}(\D)\subseteq B_{p_2}(r)$. If $f \not\equiv 0$, then there exists an inner
function $S$ and an outer function $O$ such that $f=SO$. By \cite[Proposition 4.2]{AK}, $OB\in B_{p_2}(r)$. Hence $\{z_n\}$ is a zero set of $B_{p_2}(r)$. This is a contradiction. Thus $f\equiv 0$.
\end{proof}

Next we  prove  the analytic version of Theorem 1.2, without using Theorem 5.1.

\begin{thm} Suppose $1<p<\infty$ and $0<s<1$. Let $f$ be the symbol of a bounded multiplication operator $M_f$ on $\qpsd$ space. Then $\sigma(M_f)=\overline{f(\D)}$.
\end{thm}
\begin{proof} Let $\lambda \in f(\D)$. Note that $M_f-\lambda E=M_{f-\lambda}$. Clearly,  $M_{f-\lambda}$ is not invertible. Thus $\lambda \in \sigma(M_f)$. Since $\sigma(M_f)$ is compact, we get $\overline{f(\D)}\subseteq \sigma(M_f)$.

Let $\lambda\not\in\overline{f(\D)}$. Then there exists a positive constant $C$ such that
$$
\inf_{z\in \D}|f(z)-\lambda|>C,
$$
which shows
$g(z)=\frac{1}{f(z)-\lambda}\in H^\infty$. Thus  $M_f-\lambda E$ is injective. Clearly, $f\in H^\infty$.  For any $h\in \qpsd$,  $fh\in \qpsd$. Consequently, for any $a\in \D$,
\[
\begin{split}
 \int_{\D} |g'(z)&h(z)|^p(1-|z|)^{p-2}(1-|\sigma_a(z)|)^sdA(z)
 \\
&\thickapprox  \int_{\D} \f{|f'(z)|^p|h(z)|^p}{|f(z)-\lambda|^{2p}}(1-|z|)^{p-2}(1-|\sigma_a(z)|)^sdA(z)
\\
&\lesssim \int_{\D} |f'(z)h(z)|^p(1-|z|)^{p-2}(1-|\sigma_a(z)|)^sdA(z)
\\
&\lesssim  \|fh\|^p_{\qpsd}+
\int_{\D} |f(z)h'(z)|^p(1-|z|)^{p-2}(1-|\sigma_a(z)|)^sdA(z)
\\
&\lesssim  \|fh\|^p_{\qpsd}+ \|f\|_{H^\infty}^p \|h\|^p_{\qpsd}.
\end{split}
\]
Thus $gh\in \qpsd$. Then we get $g\in M(\qpsd)$. It follows that $M_f-\lambda E$ is surjective. Hence $\lambda \not\in \sigma(M_f)$. The proof is complete.
\end{proof}
\mbox{}
\\
\noindent
\textbf{Acknowledgements:}
The work was done while G. Bao visiting the Department of Applied Mathematics and Analysis, University of Barcelona in 2015.
 He thanks the support given by the IMUB during his visit.

\end{document}